\documentclass[12pt]{amsart}
\usepackage[all]{xy}
\usepackage{amssymb}
\usepackage{amsthm}
\usepackage{hyperref}
\hypersetup{colorlinks=true,linkcolor=blue,citecolor=magenta}
\usepackage{amsmath}
\usepackage{amscd,enumitem}
\usepackage{verbatim}
\usepackage{eurosym}
\usepackage{float}
\usepackage{graphicx}
\usepackage[section]{placeins}
\usepackage{color}
\usepackage{dcolumn}
\usepackage[mathscr]{eucal}
\usepackage[all]{xy}
\usepackage{hyperref}
\usepackage{bbm}
\usepackage[textheight=8.5in, textwidth=6.7in]{geometry}
\usepackage{multirow}
\usepackage{caption}
\newtheorem*{thm*}{Theorem}
\newtheorem*{conj*}{Conjecture}

\newtheorem*{remark}{Remark}

\newtheorem{thm}{Theorem}[section]
\newtheorem{lem}{Lemma}[section]
\newtheorem{cor}[thm]{Corollary}

\newtheorem{prop}[thm]{Proposition}

\newtheorem*{rmk}{Remark}

\newcommand{\ord}{\mathrm{ord}}
\newcommand{\Z}{\mathbb{Z}}
\newcommand{\Q}{\mathbb{Q}}

\newcommand{\N}{\mathbb{N}}
\newcommand{\F}{\mathbb{F}}
\newcommand{\pil}{\pi_{\lambda,p}}
\newcommand{\pilb}{\bar{\pi}_{\lambda,p}}

\newcommand{\tor}{\mathrm{tor}}

\newcommand{\numberlist}[2][0.8\linewidth]{%
  \{\parbox[t]{#1}{\printcommalist{#2}}%
  }
\newcommand{\printcommalist}[1]{%
  \begingroup\lccode`~=`,\lowercase{\endgroup\def~}{\mathcomma\penalty0 }%
  \mathcode`,="8000
  \thinmuskip=6mu plus 6mu minus 2mu
  $#1\}.$
}
\mathchardef\mathcomma=\mathcode`,
\newtheorem*{theorem*}{Theorem}

\makeatletter
\AtBeginDocument{%
  \expandafter\renewcommand\expandafter\subsection\expandafter{%
    \expandafter\@fb@secFB\subsection
  }%
}
\makeatother

\numberwithin{equation}{section}

\makeatletter
\newcommand\footnoteref[1]{\protected@xdef\@thefnmark{\ref{#1}}\@footnotemark}
\makeatother

\begin{document}
\title{On $L$-Functions of Modular Elliptic Curves and Certain $K3$ Surfaces}
\author{Malik Amir, Letong Hong}
\address{Department of Mathematics and Statistics, McGill University}
\address{Department of Mathematics, \'Ecole Polytechnique F\'ed\'erale de Lausanne}
\email{malik.amir.math@gmail.com}
\address{Department of Mathematics, Massachusetts Institute of Technology, Cambridge, MA 02139}
\email{clhong@mit.edu}

\begin{abstract} 
Inspired by Lehmer's conjecture on the nonvanishing of the Ramanujan $\tau$-function, one may ask whether an odd integer $\alpha$ can be equal to $\tau(n)$ or any coefficient of a newform $f(z)$. Balakrishnan, Craig, Ono, and Tsai used the theory of Lucas sequences and Diophantine analysis to characterize non-admissible values of newforms of even weight $k\geq 4$. We use these methods for weight $2$ and $3$ newforms and apply our results to $L$-functions of modular elliptic curves and certain $K3$ surfaces with Picard number $\ge 19$. In particular, for the complete list of weight $3$ newforms $f_\lambda(z)=\sum a_\lambda(n)q^n$ that are $\eta$-products, and for $N_\lambda$ the conductor of some elliptic curve $E_\lambda$, we show that if $|a_\lambda(n)|<100$ is odd with $n>1$ and $(n,2N_\lambda)=1$, then
\begin{align*}
a_\lambda(n) \in \numberlist[.65\linewidth]{-5,9,\pm 11,25, \pm41, \pm 43, -45,\pm47,49, \pm53,55, \pm59, \pm61, \pm 67, -69,\pm 71, \pm 73,75, \pm79,\pm81, \pm 83, \pm89,\pm 93 \pm 97, 99}
\end{align*}
Assuming the Generalized Riemann Hypothesis, we can rule out a few more possibilities leaving 
\begin{align*}
a_\lambda(n) \in \{-5,9,\pm 11,25,-45,49,55,-69,75,\pm 81,\pm 93, 99\}.
\end{align*}

\end{abstract}

\maketitle

\section{Introduction and Statement of the Results}
In an article entitled \textit{``On certain arithmetical functions''}, Ramanujan introduced the $\tau$-function in 1916, known as the Fourier coefficients of the weight $12$ modular form $$\Delta(z)=q\prod_{n=1}^\infty (1-q^n)^{24}:=\sum_{n=1}^\infty \tau(n)q^n=q-24q^2+252q^3-1472q^4+4830q^5-...$$ where throughout $q:=e^{2\pi i z}$. It was conjectured by Ramanujan that the $\tau$-function is multiplicative and this offered a glimpse into a much more general theory known today as the theory of Hecke operators. Despite its importance in the large web of mathematics and physics, basic properties of $\tau(n)$ are still unknown. The most famous example is Lehmer's conjecture about the nonvanishing of $\tau(n)$. Lehmer proved that if $\tau(n)=0$, then $n$ must be a prime \cite{Lehmer}.  One may be interested in studying odd values taken by $\tau(n)$ or any newform coefficients. This is the question we consider as a variation of Lehmer's original speculation.

 For an odd number $\alpha$, Murty, Murty and Shorey \cite{MMS} proved using linear forms in logarithms that $\tau(n)\neq\alpha$ for all $n$ sufficiently large. However, the bounds that they obtained are huge and computationally impractical. Recently, using the theory of Lucas sequences, Balakrishnan, Craig, Ono and Tsai proved in \cite{BCO} and \cite{BCOT} that $$\tau(n)\not\in\{\pm1,\pm3,\pm5,\pm7,\pm13,\pm17,-19,\pm23,\pm37,\pm691\}.$$ Their methods apply also to newforms of even weights $k\geq 4$ and they obtained a number of further consequences. For example, they determine infinitely many spaces for which the primes $3\leq \ell\leq 37$ are not absolute values of coefficients of any newform with integer coefficients.

In this paper, we consider weights $k\leq 3$ and $L$-functions arising from weight $2$ newforms and weight $3$ newforms corresponding to modular elliptic curves and a special family of $K3$ surfaces $\{X_\lambda\}$ with Picard number $\ge 19$. We then study an infinite family $\{L(X_\lambda,s)\}$ of $L$-functions associated to these surfaces and determine some of their inadmissible coefficients.

We begin by recalling the celebrated modularity theorem. For $E/\Q$ an elliptic curve with conductor $N$ given by the equation
\begin{equation}
\label{eqn:MT}
    E:y^2=x^3+a_4x+a_6, \qquad(a_4,a_6\in\Z)
\end{equation} 
 there exists a newform $f_E(z)=\sum_{n\geq 1}a_E(n)q^n$ of weight $2$ and level $N$ for which the $p^{th}$ Fourier coefficient satisfies $a_E(p)=p+1-\#E(\F_p)$ whenever $p\nmid 2N$. The next theorem determines some odd values $|\alpha|<101$ which cannot be Fourier coefficients of $f_E(z)$. 
\begin{thm}{\label{1.1}}
Suppose that $E/\mathbb{Q}$ is an elliptic curve of conductor $N$ with a rational $2$-torsion point. For all integers $n>1$ satisfying $\gcd (n,2N)=1$, the following are true.
\begin{enumerate}
    \item If $E/\Q$ also has a rational $3$-torsion point and $3\nmid n$, then $$a_E(n)\not \in \{5, -7, 11, -13, 17, 23, -25, 35, -37,-49, -55, 65,  -73,77, -85, -91, -97\}.$$ Assuming GRH, we have $$a_E(n)\not \in \{-43, 47, 53, 59, -61, -67, 71, -73, -79, 83, 89\}.$$
    \item If $E/\Q$ also has a rational $5$-torsion point and $5\nmid n$, then $$a_E(n)\not \in \{ -11, -31, -41, -61, -71, -101\}.$$
\end{enumerate}
\end{thm}
\begin{remark}
Balakrishnan, Craig, Ono and Tsai considered in \cite{BCO} and \cite{BCOT} newforms of even weights $k\geq 4$. For weights $k<4$, their method encounters difficulties arising from equations with infinitely many integer solutions. The proof of Theorem \ref{1.1} shows how the existence of a rational $3$-torsion or a rational $5$-torsion point helps derive congruence properties that eliminate such degenerate cases.
\end{remark}
We also consider a special one-parameter family $\{X_\lambda\}$ of $K3$ surfaces\footnote{See Section \ref{ZetaFunction} for basic definitions and related results.} introduced by Ahlgren, Ono, and Penniston in \cite{AOP}, whose function field corresponds to 
\begin{equation}
    X_{\lambda}: s^2=xy(x+1)(y+1)(x+\lambda y).
\end{equation}
As the Picard number of $X_{\lambda}$ is generically $19$, they proved that the local zeta-function at primes $p$ of good reduction is given by $$Z(X_\lambda/\F_p,x)=\frac{1}{(1-x)(1-p^2x)(1-px)^{19}(1-\gamma px)(1-\gamma\pil^2x)(1-\gamma\pilb^2x)}.$$ Here, $\gamma=\left(\frac{\lambda+1}{p}\right)$ denotes the Legendre symbol. For $\lambda\in\mathbb{Q}\setminus\{0,-1\}$, Ahlgren, Ono and Penniston defined in \cite{AOP} the $L$-function of $X_\lambda$ by 
 \begin{equation}\label{link}
     L(X_\lambda ,s)=\sum_{ n\ge 1}\frac{a_{\lambda}(n)}{n^s}:=\prod_{p\nmid \lambda(\lambda+1)} \frac{1}{1-a_{\lambda}(p)p^{-s}+ p^{2-2s}},
 \end{equation}
  where $a_{\lambda}(p):= \gamma(\pil^2+\pilb^2)$ and $\pil,\pilb$ are Galois conjugates satisfying $$|\pil+\pilb|\le 2\sqrt{p}\hspace{5mm}\text{ and }\hspace{5mm}\pil\pilb=p.$$ 
  This family of $K3$ surfaces is particularly interesting. A special subset of them correspond to weight $3$ newforms with complex multiplication. By the theory of Inose and Shioda \cite{IS, Shioda}, it is known that symmetric squares of elliptic curves with complex multiplication offer $K3$ surfaces that correspond to weight $3$ CM newforms, these are the so-called \textit{singular} $K_3$ surfaces. Therefore, we seek the specializations of $X_{\lambda}$. Moreover, it is natural to consider variants of Lehmer's conjecture for all of the $\{L(X_\lambda,s)\}$.
  
  For $\lambda\in \{1, 8,1/8, -4, -1/4,-64,-1/64\}$, we denote $f_\lambda(z)=\sum_{n\ge 1}\left(\frac{D}{n}\right)a_{\lambda}(n)q^n$ to be the twist (maybe trivial) of one of the only four weight $3$ newforms given by an $\eta$-product \cite{Martin} with complex multiplication by $\Q(\sqrt{-2})$, $\Q(i)$, $\Q(\sqrt{-3})$ and $\Q(\sqrt{-7})$ respectively.
\vspace{2mm}
\begin{equation}{\label{eta}}
f_\lambda(z)=\left\{
	\begin{array}{ll}
	\vspace{2mm}
		  	\eta^2(z)\eta(2z)\eta(4z)\eta^2(8z)\otimes \chi_{-4}	& \mbox{if } \lambda=1,\\
		  	\vspace{2mm}
		  	\eta^6(4z)   & \mbox{if } \lambda=8, \\
		  \vspace{2mm}
		    \eta^6(4z)\otimes \chi_8   & \mbox{if } \lambda= \frac18, \\
		  \vspace{2mm}
		 	\eta^3(2z)\eta^3(6z)& \mbox{if } \lambda=-4,\\
	\vspace{2mm}
		 \eta^3(2z)\eta^3(6z)\otimes \chi_{-4}   & \mbox{if } \lambda=-\frac14, \\
			\vspace{2mm}
			\eta^3(z)\eta^3(7z) & \mbox{if } \lambda=-64,\\
			\vspace{2mm}
		\eta^3(z)\eta^3(7z)\otimes \chi_{-4}& \mbox{if } \lambda=-\frac{1}{64}.
	\end{array}
\right.
\end{equation}

Ahlgren, Ono and Penniston proved the following definitive theorem concerning the values $\lambda$ for which $X_\lambda$ is singular.
\begin{theorem*}[Theorem 1.2 of \cite{AOP}]\label{ClassificationTheorem}
For $\lambda\in \Q\setminus\{0,-1\}$, $X_\lambda$ is singular if and only if $$\lambda\in \{1,8,1/8,-4,-1/4,-64,-1/64\}$$
\end{theorem*}
The proof of this result relies on connections between the $K3$ surfaces $\{X_\lambda\}$, a special family of elliptic curves $\{E_\lambda\}$ defined in (\ref{EC}), and the deep work of Shioda. For these weight $3$ Hecke eigenforms, we prove the following variant of Lehmer's speculation.

\begin{thm}\label{1.2}
For all integers $n>1$ satisfying $\gcd(n, 2N_{\lambda})=1$, the following are true. 
\begin{enumerate}
\item If $\lambda=1$, then we have
\begin{align*}
a_\lambda(n)\not \in \numberlist[.71\linewidth]{\pm 1,\pm 3,5, \pm 7,\pm 9, 11, \pm13,\pm15, \pm 17, \pm 19,\pm21, \pm23,\pm25,\pm27 \pm29, \pm 31,\pm33,\pm35,\pm 37,\pm39,\pm 45,-49,\pm51,-55, \pm57,\pm63,\pm65,\pm69,-75,\pm77,\pm85,\pm87,\pm91,\pm95,\pm99}
\end{align*}
Assuming \textit{GRH}, we have 
\begin{align*}
    a_\lambda(n)\not \in \numberlist[.81\linewidth]{\pm41, \pm 43, \pm47, \pm53, \pm59, \pm61, \pm 67, \pm 71, \pm 73, \pm79, \pm 83, \pm89, \pm 97}
\end{align*}
    
\item If $\lambda\in \{8, 1/8\}$, then we have
\begin{align*}
a_\lambda(n)\not \in \numberlist[.73\linewidth]{\pm 1,\pm 3,\pm5, \pm7,-9, -11, \pm 13,\pm15, \pm 17, \pm 19,\pm21, \pm23,\pm25,\pm27, \pm29, \pm 31,\pm33,\pm35,
\pm37,\pm39,\pm 45,-49,\pm51,\pm55,\pm57,\pm63,\pm65,69,\pm 75,\pm77,\pm85,\pm87,\pm91,\pm95,-99}
\end{align*}
Assuming \textit{GRH}, we have 
\begin{align*}
    a_\lambda(n)\not \in \numberlist[.81\linewidth]{\pm 41, \pm 43, \pm47, \pm53, \pm59, \pm 61, \pm 67, \pm 71, \pm73, \pm 79, \pm 83, \pm89, \pm 97}
\end{align*}

\item If $\lambda\in \{-4, -1/4\}$, then we have
\begin{align*}
a_\lambda(n)\not \in \numberlist[.71\linewidth]{\pm 1,\pm 3,\pm5, \pm 7,\pm9,\pm 11,\pm 13,\pm15, \pm 17, \pm 19,\pm21, \pm23,-25,\pm27 \pm29, \pm 31,\pm33,\pm35,
\pm 37,\pm39,45,\pm49,\pm51,\pm55,\pm57,\pm63,\pm65,\pm69,\pm 75,\pm77,\pm85,\pm87,\pm91,\pm95,\pm99}
\end{align*}
Assuming \textit{GRH}, we have 
\begin{align*}
    a_\lambda(n)\not \in \numberlist[.81\linewidth]{\pm41, \pm 43, \pm47, \pm53, \pm59, \pm 61, \pm 67, \pm 71,\pm73, \pm79, \pm83, \pm89, \pm 97}
\end{align*}

    \item If $\lambda\in \{-64, -1/64\}$, then we have
    \begin{align*}
a_\lambda(n)\not \in \numberlist[.73\linewidth]{\pm 1,\pm 3,\pm5, \pm7,-9, \pm11, \pm 13,\pm15, \pm 17, \pm 19, \pm21,\pm 23,-25,\pm27 \pm29, \pm 31,\pm33,\pm35,
\pm37,\pm39,\pm 45,\pm49,\pm51,\pm55,\pm57,\pm63,\pm65,\pm69,-75,\pm77,\pm85,\pm87,\pm91 ,\pm95,\pm99}
\end{align*}
Assuming \textit{GRH}, we have 
\begin{align*}
    a_\lambda(n)\not \in \numberlist[.81\linewidth]{\pm41, \pm 43, \pm47, \pm 53, \pm59, \pm61, \pm 67, \pm 71, \pm 73, \pm79, \pm83, \pm89, \pm 97}
\end{align*}    
\end{enumerate}
\end{thm}

\begin{rmk}
Some of the omitted values above are in fact Fourier coefficients of these $\eta$-products.

\noindent
(i) For $\lambda=1$, we have $a_1(9)=-5,a_1(3^4)=-11,a_1(7^2)=49,a_1(27^2)=55,a_1(11^2)=75$. 

\smallskip
\noindent
(ii) For $\lambda\in \{8, \frac{1}{8}\}$, we have $a_\lambda(3^2)=9,a_\lambda(5^2)=11,a_\lambda(7^2)=49,a_\lambda(13^2)=-69,a_\lambda(3^25^2)=99$.

\smallskip
\noindent
(iii) For $\lambda\in \{-4,-\frac{1}{4}\}$, we have $a_\lambda(5^2)=25,a_\lambda(7^2)=-45$.

\smallskip
\noindent
(iv) For $\lambda\in \{-64,-\frac{1}{64}\}$, we have $a_\lambda(3^2)=9,a_\lambda(5^2)=25,a_\lambda(37^2)=75$. 
\end{rmk}
It turns out that the proof of Theorem \ref{1.2} does not rely on the modularity of $X_\lambda$. It rather depends on the pair of Galois conjugates coming from the family of elliptic curves $\{E_\lambda\}$ with conductor $N_\lambda$ defined by
\begin{equation}\label{EC}
    E_\lambda: y^2=\left(x-1\right)\left(x^2-\frac{1}{\lambda+1}\right),\qquad \lambda\in \Q\setminus\{0, -1\}.
\end{equation}
As a consequence, we obtain the following theorem. 
\begin{thm}\label{1.3}
For all $\lambda \in \Q\setminus\{0,-1\}$ and $n> 1$ satisfying $\gcd(n, 2N_\lambda)=1$, we have 
\begin{align*}
a_\lambda(n)\not \in \numberlist[0.48\linewidth]{\pm 1,\pm 3, 5, -7,-15, \pm 17, \pm 19,21, -23, -27, -29,\, \pm 31, \,33, \,37,\,-39,\,-51, 57, 69,-87}
\end{align*}
Assuming \textit{GRH}, we have 
\begin{align*}
    a_\lambda(n)\not \in \numberlist[.62\linewidth]{-41, \pm 43, -47, -53, -59, \pm 67, \pm 71, 79, -83, -89, \pm 97}
\end{align*}    
\end{thm}

As in \cite{BCO, BCOT}, our work relies on the theory of Lucas sequences. Section \ref{S2} introduces the main ideas of the theory of Lucas sequences to locate odd values in a sequence $\{a(n)\}_{n\geq 1}$. It turns out that this analysis reduces to finding integer points on some curves obtained by a two-term recurrence relation. In Section \ref{S3}, we introduce basic definitions and results related to modular elliptic curves and $K3$ surfaces, and we prove Theorems \ref{1.1},\,\,\ref{1.2} and \ref{1.3}.

\section*{Acknowledgments}
The authors are supported by the NSF (DMS-2002265), the NSA (H98230-20-1-0012), the Templeton World Charity Foundation, and the Thomas Jefferson Fund at the University of Virginia. We are grateful to Professor Ken Ono, Will Craig and Wei-Lun Tsai for the many valuable comments and discussions.

\section{Nuts and Bolts}\label{S2}
\subsection{Lucas Sequences and their primitive prime divisors}
We recall the deep work of Bilu, Hanrot and Voutier \cite{BHV} on Lucas sequences, which is central to this paper. 

A \textit{Lucas pair} $(\alpha,\beta)$ is a pair of algebraic integer roots of a monic quadratic polynomial $F(x)=(x-\alpha)(x-\beta)\in \Z[x]$ such that $\alpha+\beta$, $\alpha \beta$ are coprime non-zero integers and such that $\alpha/\beta$ is not a root of unity. To any Lucas pair $(\alpha,\beta)$ we can associate a sequence of integers $\{u_n(\alpha,\beta)\}=\{u_1=1, u_2=\alpha+\beta,\dots\}$ called \textit{Lucas numbers} defined by the following formula  
\begin{equation}
u_n(\alpha,\beta):=\frac{\alpha^n-\beta^n}{\alpha-\beta}.
\end{equation}
A prime  $\ell \mid u_{n}(\alpha,\beta)$ is a {\it primitive prime divisor of $u_n(\alpha,\beta)$} if $\ell \nmid (\alpha-\beta)^2 u_1(\alpha,\beta)\cdots u_{n-1}(\alpha, \beta)$. We call a Lucas number $u_n(\alpha,\beta)$ with $n>2$ {\it defective}\footnote{We do not consider the absence of
a primitive prime divisor for $u_2(\alpha,\beta)=\alpha+\beta$ to be   a defect.}  if $u_{n}(\alpha,\beta)$ does not have a primitive prime divisor. Bilu, Hanrot, and Voutier \cite{BHV} proved the following definitive theorem for all Lucas sequences.
\begin{thm}\label{Bilu}
Every Lucas number $u_n(\alpha,\beta)$, with $n>30,$
has a primitive prime divisor.
\end{thm}
 Tables 1-4 in Section 1 of \cite{BHV} and Theorem 4.1 of \cite{Abouzaid} offer the complete classification of these defective terms. The main arguments of our proofs will largely rely on relative divisibility properties of Lucas numbers. We now recall some of these facts\footnote{See Section 2 of \cite{BHV}.}.

\begin{prop}[Prop. 2.1 (ii) of \cite{BHV}]\label{PropA}  If $d\mid n$, then $u_d(\alpha, \beta) | u_n(\alpha,\beta).$
\end{prop}

In order to keep track of the first occurrence of a prime divisor, we define $m_{\ell}(\alpha,\beta)$ to be the smallest $n\geq 2$ for which $\ell \mid u_n(\alpha,\beta)$. We note that $m_{\ell}(\alpha,\beta)=2$ if and only if
$\alpha +\beta\equiv 0\pmod {\ell}.$
\begin{prop}[Cor. 2.2\footnote{This corollary is stated for Lehmer numbers. The conclusions hold for Lucas numbers because $\ell \nmid (\alpha+\beta)$.} of \cite{BHV}]\label{PropB} If $\ell\nmid \alpha \beta$ is an odd prime with
$m_{\ell}(\alpha,\beta)>2$, then the following are true.
\begin{enumerate}
\item If $\ell \mid (\alpha-\beta)^2$, then $m_{\ell}(\alpha,\beta)=\ell.$
\item If $\ell \nmid (\alpha-\beta)^2$, then $m_{\ell}(\alpha,\beta) \mid (\ell-1)$ or $m_{\ell}(\alpha,\beta)\mid (\ell+1).$
\end{enumerate}
\end{prop}

\subsection{Some Diophantine criteria}{\label{sec2.2}}
Consider a sequence $\{a(n)\}_{n\geq 1}$ satisfying the following properties which we will refer to as ($\dagger$):
\begin{enumerate}
    \item $a(1)=1$ and $|a(n)|\neq 1$ for all $n>1$. \label{prop1}
    \item $a(nm)=a(n)a(m)$ whenever $\gcd(n,m)=1$. \label{prop2}
    \item $a(p)a(p^m)=a(p^{m+1})+\chi(p)p^{k-1}a(p^{m-1})$, \label{prop3}
\end{enumerate}
where $\chi$ is a real quadratic character modulo $N\in \N$. The reader may recognize in $(\dagger)$ the basic identities satisfied by the Fourier coefficients of normalized newforms (see for example \cite{CoSt,Ono}).
\begin{prop}\label{Newforms} Suppose that $f(z)=q+\sum_{n=2}^{\infty}a(n)q^n\in S_{k}(\Gamma_0(N),\chi)$ is a normalized newform with nebentypus $\chi$.
Then the following are true.
\begin{enumerate}
\item If $\gcd(n_1,n_2)=1,$ then $a(n_1 n_2)=a(n_1)a(n_2).$
\item If $p\nmid N$ is prime and $m\geq 2$, then $$a(p^m)=a(p)a(p^{m-1})-\chi(p)p^{k-1}a(p^{m-2}).
$$
\item If $p\nmid N$ is prime and $\alpha_p$ and $\beta_p$ are roots of $F_p(x):=x^2-a(p)x+\chi(p)p^{k-1},$ then
$$
   a(p^m)=u_{m+1}(\alpha_p,\beta_p)=\frac{\alpha_p^{m+1}-\beta_p^{m+1}}{\alpha_p-\beta_p}.
$$   
Moreover, we have the Deligne's bound $|a(p)|\leq 2p^{\frac{k-1}{2}}$.
\end{enumerate}
\end{prop}
Note that \textit{(3)} defines a Lucas sequence for a fixed prime $p$. This is an example of a two-term recurrence relation satisfying $(\dagger)$. Hence, any function $a(n)$ satisfying $(\dagger)$ defines such a sequence for each prime $p\nmid N$ for some $N$, i.e. $\{a(p),a(p^2),a(p^3),...\}$. It follows that our results can be applied to suitable $L$-functions. In this more general case, the Lucas sequences we consider are those obtained from $$u_{n}(\alpha_p,\beta_p):=a(p^{n-1})=\frac{\alpha_p^{n}-\beta_p^{n}}{\alpha_p-\beta_p},$$ where $p$ is a prime, $\alpha_p$ and $\beta_p$ are roots of \begin{equation}
    F_p(x):=x^2-Ax+B:=x^2-a(p)x+\chi(p)p^{k-1},
\end{equation} 
and $|a(p)|\leq2p^{\frac{k-1}{2}}$. Suppose $u_n(\alpha,\beta)$ is such a sequence for the remainder of this section.
\begin{lem}\label{nondefect}
Assume $a(p)$ is even for primes $p\nmid 2N$. The only defective odd values $u_d(\alpha_p,\beta_p)$ are given in Table 1 by $$(d,A,B,k)\in\{(3,\pm 2,3,2),(5,\pm 2,11,2)\},$$ and in rows 1-3 of Table 2.
\end{lem}
\begin{proof}
This follows by a careful analysis of Tables \ref{table1} and \ref{table2} in the Appendix. In Table \ref{table1}, it is clear that the only cases are the above ones. In Table \ref{table2}, we first eliminate the possibilities of rows 5-10 as the values of $u_d(\alpha_p,\beta_p)$ are even (for row 8, $u_6$ is even since $m$ is). Then we eliminate row 4 as $\pm m=a(p)$ is odd.
\end{proof}
\begin{remark}
For $u_3(\alpha_p,\beta_p)=\varepsilon 3^r$, we have $3\nmid a(p)$ and so $u_3(\alpha,\beta)$ is the first occurrence of $3$ in the sequence. However, the term is still defective according to definition. We note that Table \ref{table2} row 3 is for even weight only, and when $k=3$, all listed cases do not apply.
\end{remark}
The following theorem and its proof are inspired from Theorem 1.1 of \cite{BCOT}.
\begin{thm}{\label{thm2.5}}
Suppose $\{a(n)\}_{n=1}^{\infty}$ satisfies $(\dagger)$. If $\ell$ is an odd prime and $|a(n)|=\ell^m$ for $m\ge 1$, then for all primes $p\mid n$, $\ord _p(n)+1$ is $4$ or a prime that divides $ \ell(\ell^2-1)$. 
\end{thm}
\begin{proof}
Since $\{a(n)\}_{n=1}^{\infty}$ satisfies the two-term linear recurrence relation presented in $(\dagger)$, we have that the numbers $u_n(\alpha,\beta)=a(p^{n-1})$ form a Lucas sequence in weight $k$. If $|a(n)|=\ell^m$, then by multiplicativity we get for every prime $p\mid n$, $|a(p^{d-1})|=\ell^j$ for some $j\leq m$, where $d-1=\ord_p(n)$. Suppose $u_d=a(p^{d-1})=\pm \ell^j$. Suppose $d$ is not a prime. There exists $1<c<d$ such that $c\mid d$, then $u_c=\ell^{j_0}$ for some $1\le j_0\le j$ since $u_c$ cannot equal $1$. In such cases, $u_d$ is defective.

If we analyze Table \ref{table1}
and \ref{table2} in the Appendix, then the only defective values that are also a power of odd prime happens when $d$ is a prime, $\pm 1$ appears in the sequence after the first term, or when $d=4$ in row $4$ of Table \ref{table2}. Only the last case is possible, and combining Proposition \ref{PropB} we have the desired result.
\end{proof}

\subsection{Integer points on some curves}{\label{IPSC}} Using Theorem \ref{thm2.5}, studying $a(n)=\pm\ell$ is equivalent to studying $a(p^{d-1})=\pm\ell$ for odd primes $d|\ell(\ell^2-1)$ or $d=2$. From the two-term recurrence relation satisfied by $a(p^n)$, these equalities reduce to the search of integer points on polynomial curves. We make this statement precise now. 

Let $D$ be a non-zero integer. A polynomial equation of the form $F(X,Y)=D$, where $F(X,Y)\in \Z[X,Y]$ is a homogeneous polynomial, is called a \textit{Thue equation}. We will consider those equations arising from the series expansion of 
\begin{equation}
\label{Thue}
\frac{1}{1-\sqrt{Y}T+XT^2}=\sum_{m=0}^\infty F_m(X,Y)\cdot T^m=1+\sqrt{Y}\cdot T+(Y-X)\cdot T^2+\dots
\end{equation}
\begin{lem}
 If $a(n)$ satisfies $(\dagger)$, and $p$ is a prime, then
 $$F_{2m}\left(\chi(p)p^{k-1},a(p)^2\right)=a(p^{2m}).$$
\end{lem}
Hence, solving $a(p^{2m})=\pm\ell$ boils down to computing integer solutions to the equation $F_{2m}(X,Y)=\pm\ell$.

Methods for solving Thue equations are implemented in Sage \cite{github}, Magma, and are best suited for $m\geq 3$. For $m=1,2$, these equations often have infinitely many solutions and we must use extra information to infer finiteness. For weight 3, the curve corresponding to $m=1$ is $(p,a(p))\in C^\pm_2:y^2=\pm x^2+\ell,\chi(p)=\pm 1$ and the curve corresponding to $m=2$ is $(p,a(p))\in C^\mp_4:y^4\mp 3x^2y^2+x^4=\ell,\chi(p)=\pm1$. This forces the solutions of $F_2(X,Y)=\pm\ell$ and $F_4(X,Y)=\pm\ell$ to be squares and we get the following lemma.
\begin{lem}
The curves $C^{\pm}_2$ and $C^\mp_4$ have finitely many solutions. 
\end{lem}
\begin{proof}
This is clear by factorization of the two algebraic equations.
\end{proof}

For larger $m$, we have the following theorem, which tells us that for Lucas sequences of weight $3$, we only need to consider curves $C_2^{\pm}$ and $C_4^{\pm}$ for prime values $\pm\ell$.
\begin{lem}{\label{NoThue}}
For $7\leq\ell\leq 97$ a prime, $F_{2m}(X,Y)=\pm\ell$ has no solutions for which both $X$ and $Y$ are squares.
\end{lem}
\begin{proof}
This follows from Tables \ref{thuetable} and \ref{thueGRHtable} in the Appendix. 
\end{proof}

To compute solutions to these curves, we use Maple and the command \textit{isolve} for Diophantine equations, Wolfram$|$Alpha or \cite{OLSolver}.
\section{Proof of Theorems \ref{1.1},\ref{1.2},\ref{1.3}}\label{S3}

\subsection{Proof of Theorem \ref{1.1}}

We will make use of the two following fundamental results.
\begin{thm}[Wiles et al., 1994]\label{Modularity}
If $E / \Q$ is an elliptic curve with conductor $N$, then $E$ is associated to a newform $f_E = \sum_{ n\geq 1} a_E(n) q^n \in S_2(\Gamma_0(N))\cap \Z[[q]]$. For all primes $p\nmid 2N$, we have
$$ a_E(p) = p+1 - \#E(\F_p), $$
where $\#E(\F_p)$ denotes the number of $\F_p$-points of the elliptic curve.
\end{thm}

The following theorem of Mazur classifies the possible torsion subgroups of $E/\Q$.
\begin{thm}[Mazur, 1977]{\label{Mazur}}
If $E/\Q$ is an elliptic curve, then $$E_{\tor}(\Q)\in \{\Z/N\Z\ |\ 1\le N\le 10\ \mathrm{or}\ N=12\}\cup\{\Z/2\Z \times \Z/N\Z \ |\ N=2,4,6,8\}.$$
\end{thm}
Furthermore, recall that if $E / \Q$ is an elliptic curve with good reduction at $p\nmid m$ for some $m$, then the reduction map 
\begin{equation}\label{lemma:injectivityoftorsion}
    \pi \colon E(\Q) \to E(\F_p),
\end{equation}
is injective when restricted to $m$-torsion \cite{Silverman}. As a consequence, we get the following results.
\begin{lem}\label{lemma1}\label{Initial congruence}
Suppose that $E / \Q$ is an elliptic curve and that $\ell\mid\#E_{\tor}(\Q)$. Then for all primes $p\nmid 2\ell N$, we have
$$a_E(p^d) \equiv 1+p+p^2+\cdots +p^d \mod \ell. $$
\end{lem}

\begin{proof}
Since our newform satisfies $a_E(p^m)=a_E(p)a_E(p^{m-1})-pa_E(p^{m-2})$ we get that 
$$(px^2-a_E(p)x+1)\Big(\sum_{m\geq0}a_E(p^m)x^m\Big)=1\iff \sum_{m\geq0}a_E(p^m)x^m=\frac{1}{px^2-a_E(p)x+1}.$$ Expanding the RHS as a power series we obtain $$\frac{1}{px^2-a_E(p)x+1}=\frac{1}{(\alpha_px-1)(\beta_px-1)}=\sum_{m\geq 0}\frac{\alpha_p^{m+1}-\beta_p^{m+1}}{\alpha_p-\beta_p}x^m,$$ and hence $$a_E(p^m)=\frac{\alpha_p^{m+1}-\beta_p^{m+1}}{\alpha_p-\beta_p}.$$

Since $p\nmid 2N$, $E$ has good reduction at $p$. Therefore, by Theorem \ref{Modularity}, $a_E(p)=p+1-\#E(\F_{p})$. Because $\ell\mid\#E_{\tor}(\Q)$ and $\ell \neq p$, it follows from Theorem \ref{Mazur} that $\ell\mid\#E(\F_p)$. Reducing mod $\ell$ gives $a_E(p)\equiv p+1\bmod \ell$. Hence $\alpha_p\equiv 1,\beta_p\equiv p\bmod \ell$ and the result follows.
\end{proof}

\begin{lem}
If $E/\Q$ has a rational $2\ell$-torsion point with $\ell=3,5$, then for all $\gcd(n,2N)=1$, we have $|a_E(n)|\neq 1$.
\end{lem}
\begin{proof}
The various possibilities for $|a_E(p^{d-1})|=1$ are in Tables $1$ and $2$ of the appendix. With $a_E(p)$ even, we can rule out some sporadic terms. For the weight $2$ case, we shall only have $u_3(\alpha,\beta)=1, p=3, a_E(p)=\pm 2$ or $u_3(\alpha,\beta)=-1, p=a_E(p)^2+1$. If it is the latter, then by lemma \ref{Initial congruence} we have $p\equiv p^2+2p+2\bmod \ell$. However for $\ell=3,5$, $p^2+p+2\not \equiv 0\bmod \ell$ for any choice of $p$. The former case cannot happen if $5\mid \#E_{\tor}(\Q)$ because $3+1\not \equiv \pm 2\bmod 5$. If $3\mid \#E_{\tor}(\Q)$ then this tells us $3\mid N$, so $3\nmid n$ and hence $p$ cannot take the value of $3$.
\end{proof}

\begin{proof}[Proof of Theorem \ref{1.1}]
The existence of a rational $2$-torsion point and the injectivity of torsion map tell us that for $p\nmid 2N$, $a_E(p^{d-1})$ odd if and only if $d$ is odd. Let $\ell=3$ or $5$ dividing $\# E_{\text{tor}}(\Q)$. Given an odd prime $\ell_1\equiv -1\bmod \ell$, one can see from Lemma \ref{Initial congruence} that in order to have $a_E(p^{d-1})=\ell_1\equiv -1\bmod \ell$ given that $d$ is odd, then $p\equiv 1\bmod \ell$ and $d\equiv -1\bmod \ell$. Hence, Proposition \ref{PropB} tells us that we only need to look for the odd primes $d$ that are congruent to $-1\bmod \ell$ and divide $\ell_1(\ell_1^2-1)$. This rules out the possibility that $d=3$ and so we can avoid running into the Thue equation with infinitely many integer points.

Some inadmissable composite values $\alpha$ appear in the first part of Theorem \ref{1.1}, where a $3$-torsion point exists. For these listed composite numbers, they satisfy the following properties: we cannot properly factorize them into two admissable odd values. Therefore, $a_E(n)=\alpha$ implies that $n$ only has a unique prime factor, i.e. $n=p^{d-1}$ for some $d$. Indeed, if $p^u\mid n$ and $n=p^un_0$ where $n_0$, coprime to $p$, is non-trivial, then $a_E(p^u)a_E(n_0)=\alpha$ and this is a contradiction to the assumption. We then consider whether $d$ is allowed to be a composite value. If so, say prime $c\mid d$, then $u_c=a_E(p^{c-1})$ divides $\alpha$. We analyze what could be the possible values of $u_c$ and also what is the primitive prime divisor of $u_d$ to infer what $c$ and $d$ are. In the process, note that if $\alpha$ is the product of at most two powers of odd primes, then $d$ cannot be properly factorized into two different numbers, i.e. it is either a square or a prime. This is because if $d=c_1c_2$ then $u_{c_1}$, $u_{c_2}$, $u_d$ each occupy a different prime factor, assuming that they are non-defective. However, $u_d=\alpha$, a multiple of the first two terms, only has at most $2$ prime divisors. (One may notice that in our case here, the only possible defective value occurs when $p=11, a_E(p)=\pm 2$ at $u_5=5$ in Table \ref{table1} of the Appendix. Theorem \ref{Modularity} tells us that $12\equiv \pm 2\bmod 3$ since $3\mid E_{\text{tor}}(\Q)$, but the congruence is not possible.) Note also that none of the listed composite values can have three distinct prime factors, since the minimal possible one, $3\times 5\times 7=105$ is already larger than $100$.

With these results and the assumption that $c\mid d$ and $c$ is prime, we can greatly reduce the number of cases to consider. Then it remains to run through the possible solutions to Thue equations. We hereby leave the case-by-case calculations to the reader.
\end{proof}
\subsection{Proof of Theorem \ref{1.2} and \ref{1.3}}

We state the relationship between symmetric square $L$-functions of weight $2$ newforms and $L$-functions of weight $3$ newforms. We explain the relationship between the $K3$ surfaces $X_\lambda$ and the elliptic curves $E_\lambda$ using the work of Ahlgren, Ono, Penniston \cite{AOP}, and Inose and Shioda \cite{IS,Shioda}.  
\subsubsection{Symmetric square $L$-functions and the relation between $\{X_\lambda\}$ and $\{E_\lambda\}$}{\label{ZetaFunction}}\hfill

\vspace{3mm}
Recall from \cite{AOP} the family of $K3$ surfaces $\{X_\lambda \}$ over $\Q$ to be the double cover of the projective plane branched along the smooth sextic curve $$U_\lambda : XYZ(Z+\lambda Y)(Y+Z)(X+Z)=0.$$ According to Section 3 of \cite{AOP}, if $\psi:X_\lambda\to \mathbb{P}^2$ denotes the natural map, then $X_\lambda\setminus \psi^{-1}(U_\lambda)$ is a smooth affine surface of equation $s^2=xy(x+1)(y+1)(x+\lambda y)$ defining the function field of $X_\lambda$. This family corresponds to elliptic curves $\{E_\lambda\}$ with conductor $N_\lambda$ 
by $$E_\lambda : y^2=(x-1)\left(x^2-\frac{1}{\lambda+1}\right).$$

Consider $g(z)$ and $f(z)$ to be respectively a weight $2$ and a weight $3$ newform. We have by definition that $$L(g,s)=\prod_p\frac{1}{1-b(p)p^{-s}+p^{1-2s}},$$ is the product of local Euler factor at $p$. We can rewrite the denominator as $(1-\pi_{p}p^{-s})(1-\bar{\pi}_{p}p^{-s})$ where $\pi_{p},\,\bar{\pi}_{p}$ are the roots of $F_p(x)=x^2-b(p)x+p$. Furthermore, the symmetric square $L$-function of $g$ is given by
\begin{equation}
    \begin{aligned}
    L(\mathrm{Sym}^2(g),s)&=\prod_p\frac{1}{(1-p^{1-s})(1-\pi_p^2p^{-s})(1-\bar{\pi}_p^2p^{-s})}\\&=\zeta(s-1)\prod_p\frac{1}{1-(b(p)^2-2p)p^{-s}+p^{2-2s}},
    \end{aligned}
\end{equation}
where $b(p)^2-2p=\pi_p^2+\bar{\pi}_p^2$. The $L$-function of $f$ is given by $$L(f,s)=\prod_p\frac{1}{1-a(p)p^{-s}+\chi(p)p^{2-2s}}.$$
 
The denominator of the local $p$ factor of the symmetric square $L$-function of $g$ has the shape of the local $p$ factor of the $L$-function of $f$.

Suppose now that $g_\lambda(z)=\sum_{n\geq 1}b_\lambda(n)q^n$ is the newform of weight 2 attached to $E_\lambda$ and that $f_\lambda(z)=\sum_{n\geq 1}a_\lambda(n)q^n$ is the newform of weight $3$. From the local zeta function of $X_\lambda$ at primes $p$ with $\ord_p(\lambda(\lambda+1))=0$, proved in Section 4 of \cite{AOP}, we have $$Z(X_\lambda/\F_p,x)=\frac{1}{(1-x)(1-p^2x)(1-px)^{19}(1-\gamma px)(1-\gamma\pil^2x)(1-\gamma\pilb^2x)}.$$ By (29) of \cite{AOP}, $X_\lambda$ is singular (also known as ``modular'') if $$L(X_\lambda,s):=\prod_{p\nmid \lambda(\lambda+1)} Z^*(X_\lambda/\F_p,s)=\sum_{n\geq 1}\frac{a_\lambda(n)}{n^s},$$ where $$Z^*(X_\lambda/\F_p,s):=\frac{1}{(1-\gamma\pil^2p^{-s})(1-\gamma\pilb^2 p^{-s})}.$$ Expanding yields 
\begin{equation}{\label{mainformula}}
a_\lambda(p)=\gamma(\pil^2+\pilb^2)=\gamma(b_\lambda(p)^2-2p).  
\end{equation}
This justifies the use of $\lambda$ for the parametrization of $X_\lambda$ and $E_\lambda$ but also how they are related as was observed by Inose and Shioda \cite{IS}. Their result states in the case of a $K3$ surface of Picard number 20 that its $L$-function corresponds to the symmetric square of the $L$-function of an elliptic curve with complex multiplication up to simple terms. Namely, the Picard number jumps to $20$ precisely for the values $\lambda$ in (\ref{eta}).

Notice that we haven't used the modularity of $X_\lambda$ to define its $L$-function. We only used the pair of Galois conjugates $\pil$ and $\pilb$ obtained from $g_\lambda$. It follows that for all $\lambda\in \Q\setminus\{0,-1\}$, we can apply Deligne's bound to $b_\lambda(p)$ and define a Lucas sequence $u_n(\alpha_{\lambda,p},\beta_{\lambda,p})=a_\lambda(p^{n-1})$ from $F_p(x)=x^2-a_\lambda(p)x+p^2$. It is thus natural to consider for all $\lambda\in \Q\setminus \{0,-1\}$ the family of $L$-functions $\{L(X_\lambda,s)\}$ given by $$L(X_\lambda,s)=\sum_{n\geq 1}\frac{a_\lambda(n)}{n^s}:=\prod_{p\nmid \lambda(\lambda+1)}\frac{1}{1-a_\lambda(p)p^{-s}+p^{2-2s}}.$$ 
In the case of the seven exceptional values $\lambda$, the relation between the two aforementioned families is deeper. The proof of Theorem 1.2 of \cite{AOP} is dictated by the complex multiplication underlying $E_\lambda$ and the fact that for a fixed prime $\ell$, there are no sets of primes $p$ with positive density satisfying (\ref{mainformula}) modulo $\ell$ \cite{OS,Ribet}, as another way of seeing Inose and Shioda's theorem.

We now proceed with preliminary results based on (\ref{mainformula}).
\begin{lem}{\label{TM2}}
For $\lambda\in \mathbb{Q}\setminus \{0, -1\}$, $a_\lambda(p)$ is even for $p\nmid 2N_\lambda$. Furthermore, we have that $a_\lambda(p^d)$ is odd if and only if $d$ is even.
\end{lem}
\begin{proof}
Since $p$ is a prime of good reduction, the map from $E_{\lambda}^{\tor}(\Q)$ to $E_\lambda(\mathbb{F}_p)$ is injective. As $(1,0)$ is a $2$-torsion point for $E_\lambda$, we obtain that $2\mid \# E_{\lambda}^{\tor}(\mathbb{Q})$ and it follows that $\#E_\lambda(\mathbb{F}_p)$ is even. Therefore, for $p$ an odd prime, we obtain
$$a_{\lambda}(p)=\gamma(b_\lambda(p)^2-2p)\equiv b_\lambda(p)\equiv p+1-\#E_\lambda(\mathbb{F}_p)\bmod 2.$$
The second assertion is immediate from the recurrence relation satisfied by $a_\lambda(n)$.
\end{proof}

\begin{proof}[Proof for Theorem \ref{1.2} and \ref{1.3}]
We begin with the following observations.

\begin{lem}{\label{mainlem}}
Let $\lambda\in \Q \setminus\{0,-1\}$. Then $|a_\lambda(n)|\neq 1$ for all $\gcd(n,2N_\lambda)=1$. Furthermore, for $\ell_1$ an odd prime, we have that $a_\lambda(n)=\pm\ell_1$ reduces to $u_d(\alpha_p,\beta_p)=a_\lambda(p^{d-1})=\pm\ell_1$ where $d$ is an odd prime. 
\end{lem}
\begin{proof}
Let $\gcd(n,2N_\lambda)=1$, then this reduces by $(\dagger)$ to $|a_\lambda(p^{d-1})|=1=u_{d}(\alpha_p,\beta_p)$ for $d\geq 3$. Indeed, $u_2(\alpha_p,\beta_p)$ is even for $p\nmid 2N_\lambda$. Since $|u_{d}(\alpha_p,\beta_p)|= 1$ is a defective value, it suffices to look at Tables 1 and 2 but none of the cases apply. Hence, from the proof of Theorem \ref{thm2.5}, we get that $d$ is an odd prime.  
\end{proof}

\begin{cor}{\label{ConsThue}}
For any odd prime $\ell$ and $\gcd(n,2N_\lambda)=1$, we have that $a_\lambda(n)=\pm\ell_1$ implies $n=p^2$ or $n=p^4$ for some prime $p$.
\end{cor}
\begin{proof}
The corollary follows immediately from Lemma \ref{NoThue} and Lemma \ref{TM2}. 
\end{proof}
We now complete the proof for prime values. By Proposition \ref{PropB}, $d|\ell_1(\ell_1^2-1)$ and it thus suffices to list all prime divisors of $\ell_1(\ell_1^2-1)$. From Corollary \ref{ConsThue}, we consider only $d=3$ and $d=5$. Finally, for $d=5$, one uses Tables $5,6,7$. For $d=3$, one simply solves the curves $C^\pm_2=\ell_1$ and $C^\pm_2=-\ell_1$. In the special case of the 7 exceptional values of $\lambda$, more information is available through their complex multiplication and the modular forms database \cite{LMFDB}. Whenever we have a possible solution $(p,a_\lambda(p))=(x,y)$, it suffices to evaluate $a_\lambda(p)$ for each of the four $\eta$-products and see if the coefficient matches the solution. 

We now turn to composite values which consists of a case by case analysis of $\alpha=\pm pq,\pm p^2q$ and $\pm\ell_1^m$ where $p,q$ and $\pm \ell_1$ are odd primes. As in the proof of Theorem \ref{1.1}, $a_\lambda(n)=\alpha$ cannot be properly factorized into two admissible odd values (except for the case $\alpha=25,\,\lambda=1$ in Theorem \ref{1.2} that will be treated differently below). This reduces to consider $a(p^{d-1})=\alpha$ and recall in Theorem \ref{1.1}'s proof that $d$ is a square or an odd prime. The claim also holds here, and the only odd square possibly dividing $\ell_1(\ell_1^2-1)$ is $9$. This considerably reduces the set of candidates for $d$ and we thus omit the case by case calculations for the values of $\alpha$. Additionally, for Theorem \ref{1.2}, we can verify \cite{LMFDB} if a solution to the Thue equation describes a Fourier coefficient of $f_\lambda$ since we have located precisely where it must lie in the $q$-expansion. If it does not correspond to a Fourier coefficient, then we rule out this value.

Suppose now that $\alpha=25$ for $\lambda=1$. There is a possibility for $n$ to not be a prime power and it is given by $n=p^{d-1}q^{s-1}$ where $a_\lambda(n)=25=-5\cdot -5=a_\lambda(p^{d-1})a_\lambda(q^{s-1})$, for $p,q$ distinct odd primes. This means that $a_\lambda(p^{d-1})=-5=a_\lambda(q^{s-1})$ but $-5$ only happens at $n=3^2$. Thus $p=q$ and this is not allowed. Hence we get that $a(p^{d-1})=25$. Lemma \ref{nondefect} gives us non-defectiveness and tells us that $d$ cannot have any non-trivial factors, so $d=3, 5$. The two corresponding equations have no valid solution for $f_1(z)$. 
\end{proof}
\vspace{3mm}
\section{Appendix}

\setlength{\tabcolsep}{3pt} 
\renewcommand{\arraystretch}{1.25}
\begin{center} 
\begin{table}[!ht]
\begin{tabular}{|c|c|}
\multicolumn{2}{c}{} \\ \hline
$(A,B)$ & Defective $u_n(\alpha, \beta)$ \\ 
\hline
\hline
$(\pm 1,2^1)$ & $u_5 = -1$, $u_7 = 7$, $u_8 = \mp 3$, $u_{12} = \pm 45$, \\ & $u_{13} = -1$, $u_{18} = \pm 85$, $u_{30} = \mp 24475$ \\ \hline
$(\pm 1,2^2)$ & $u_5 = 5$, $u_{12} = \mp 231$ \\ \hline
$(\pm 1,3^1)$ & $u_5 = 1$, $u_{12} = \pm 160$ \\ \hline
$(\pm 1, 5^1)$ & $u_7 = 1$, $u_{12} = \mp 3024$ \\ \hline
$(\pm 2, 3^1)$ & $u_3 = 1$, $u_{10} = \mp 22$ \\ \hline
$(\pm 2,7^1)$ & $u_8 = \mp 40$ \\ \hline
$(\pm 2, 11^1)$ & $u_5 = 5$ \\ \hline
$(\pm 3, 2^2)$ & $u_4=\pm3$ \\ \hline
$(\pm 3, 2^3)$ & $u_3 = 1$ \\ \hline
$(\pm 4, 5^1)$ & $u_6=\pm 44$\\ \hline
$(\pm 5, 2^3)$ & $u_6 = \pm 85$ \\ \hline
$(\pm 5, 7^1)$ & $u_{10} = \mp 3725$ \\ \hline
\end{tabular}
\medskip
\captionof{table}{\textit{Sporadic examples of defective  $u_n(\alpha, \beta)$.}}
\label{table1}
\end{table}
\end{center}

\begingroup
\setlength{\tabcolsep}{3pt} 
\renewcommand{\arraystretch}{2.5}
\begin{center} 
\begin{table}[!ht]
\begin{small}
\begin{tabular}{|c|c|c|}
\multicolumn{2}{c}{}\\ \hline
$(A,B)$ & Defective $u_n(\alpha, \beta)$ & Constraints on parameters \\ \hline \hline
$(\pm m, p)$ & $u_3 = -1$ & $p = m^2+1$ \\ \hline
$(\pm m, p^{k-1})$ &
$u_3 = \varepsilon 3^r$ &
$\begin{aligned} &p^{k-1}=m^2-3^r\varepsilon\ \text{with}\ 3\nmid m, r>0\end{aligned}$\\ \hline
$(\pm m, -p^{2k-1})$ &
$u_3 = 3^r$ &
$\begin{aligned} &p^{2k-1}+m^2=3^r\ \text{with}\ 3\nmid m, r>0\end{aligned}$\\ \hline
$(\pm m, p^{k-1})$ & $u_4 = \mp m$ & $2p^{k-1}=m^2+1$\\ \hline
$(\pm m, p^{k-1})$ & $u_4 = \pm 2{\color{black}\varepsilon}m$ & $\begin{aligned}\ \ \  2p^{k-1}=m^2-2\varepsilon
\end{aligned}$ \\ \hline

$(\pm m, p^{k-1})$ & 
$u_6 = {\color{black}\pm (-2)^rm(2m^2+(-2)^r)/3}$ &
$\begin{aligned}
&\ \ \ \ 3 p^{k-1}=m^2-(-2)^r,\ r>0
\end{aligned}$\\ \hline

$(\pm m, -p^{2k-1})$ & 
$u_6 = {\color{black}\pm (-2)^rm(2m^2+(-2)^r)/3}$ &
$\begin{aligned}
& 3p^{2k-1}+m^2=2^r,\ r>0
\end{aligned}$\\ \hline
$(\pm m, p^{2k-1})$ & ${\color{black}u_6=\pm \varepsilon m(2m^2+3\varepsilon)}$ & $3p^{2k-1}=m^2-3\varepsilon$\\ \hline
$(\pm m, p^{2k-1})$ & $u_6 = \pm 2^{r+1}{\color{black}\varepsilon} m(m^2 + 3 {\color{black}\varepsilon} \cdot 2^{r-1}) $ &$\begin{aligned} 3 p^{2k-1}=m^2-3\varepsilon\cdot 2^r
 \text{ with } r>0 \end{aligned}$ \\ \hline

 $(\pm m, -p^{2k-1})$ & $u_6 = \pm 2^{r+1}{\color{black}} m(m^2 + 3 {\color{black}} \cdot 2^{r-1}) $ &$\begin{aligned} 3 p^{2k-1}+m^2=3\cdot 2^r
 \text{ with } r>0 \end{aligned}$ \\ \hline

\end{tabular}
\end{small}
\medskip
\captionof{table}{\textit{Parameterized families of defective $u_n(\alpha, \beta)$.}}
\label{table2}
\end{table}
\end{center}
\endgroup

\begingroup
\setlength{\tabcolsep}{5pt} 
\renewcommand{\arraystretch}{1.9}
\begin{center} 
\begin{table}[!ht]
\begin{tabular}{|c|c|}
\multicolumn{2}{c}{} \\ \hline
$(d, D)$ & Integer Solutions to $F_{d-1}(X,Y)=D$  \\ \hline \hline
$(7, \pm 7)$ & $(\pm 1, \pm 4), ( \pm 2, \pm 1), (\mp 3, \mp 5)$\\ \hline
\multirow{2}{3.5em}{$(7,  \pm 13)$} & $( \pm 3,  \pm 10), (\pm 2,  \pm 7), ( \pm 3, \pm 4), (\pm 4, \pm 1),$ \\ & $( \pm 3,  \pm 1), (\mp 1, \pm 1), (\mp 2, \mp 5), (\mp 5,\mp 8), (\mp 7, \mp 11)$ \\ \hline
$(7, \pm 29)$ & $\begin{aligned}(\mp 6, \mp 1), (\mp 5, &\mp 16), (\mp 4, \mp 7), (\pm 1, \pm 5), \\&(\pm 3, \pm 2), (\pm 11, \pm 17)\end{aligned}$\\ \hline
$(11, \pm 11), (19, \pm 19),$  &\multirow{2}{3.5em}{$( \pm 1,\pm 4)$}\\ 
$(23, \pm 23), (31, \pm 31)$ &\\ \hline
$(11, \pm 23)$ & $( \pm 3, \pm 2), ( \pm 2, \pm 1), (\mp 2,\mp 3)$\\ \hline
$(13, 13), (17, 17), (29, 29), (37,37)$ & $(-1, -4), (1,4)$\\ \hline
$(13, -13),  (17, -17),$  & \multirow{2}{3.5em}{$\varnothing$}\\ 
$(29, -29), (37,-37)$ &\\ \hline
$(19, \pm 37)$ & $(\mp 2, \mp 5)$ \\ \hline
\end{tabular}
\medskip
\captionof{table}{\textit{Solutions to the Thue equations where $D=\pm \ell$ and $7\leq \ell \leq 37$} \cite{BCOT}. }
\label{thuetable}
\end{table}
\end{center}
\endgroup

\begingroup
\setlength{\tabcolsep}{5pt} 
\renewcommand{\arraystretch}{1.9}
\begin{center} 
\begin{table}[!ht]
\begin{tabular}{|c|c|}
\multicolumn{2}{c}{} \\ \hline
$(d, D)$ & Integer Solutions to $F_{d-1}(X,Y)=D$  \\ \hline \hline
$(7, \pm 41) $ & $(\mp 3, \mp 7), (\mp 1, \pm 2), (\pm 4, \pm 5)  $\\ \hline
$\begin{aligned}&(41, 41), (53, 53), (61,61), \\
&(73, 73), (89,89), (97, 97)\end{aligned}$ & $(-1, -4), (1,4)$ \\ \hline
$(41, -41), (23, \pm 47), (13, 53), (53,-53), (29, \pm 59),$ &  \multirow{3}{3.5em}{$\varnothing$} \\ 
$(31, \pm 61), (61, -61), (17,-67), (37, \pm 73), (73,-73),$ & \\
$(13,-79), (41, \pm 83), (89,-89), (97,-97)$ & \\   \hline
$(7,  \pm 43)$ & $(\mp 3, \mp 8), (\mp 2, \pm 1), (\pm 5, \pm 7) $\\ \hline
$(11, \pm 43)$ & $(\mp 3, \mp 5), (\pm 2, \pm 5)$\\ \hline
$(43, \pm 43), (47, \pm 47), (59, \pm 59), (67,\pm 67),$ &  \multirow{2}{3.5em}{$( \pm 1,\pm 4)$}\\
$(71, \pm 71), (79, \pm 79), (83, \pm 83)$ &  \\ \hline
$(13, -53), (17,67)$ & $ (-2,-3), (2, 3)$\\ \hline
$(11, \pm 67)$ & $(\mp 7, \mp 12), (\mp 3, \mp 11), (\mp 2, \mp 7)$ \\ \hline

\multirow{2}{3.5em}{$(7,\pm 71)$} & $(\mp 16, \mp 25), (\mp 5, \mp 9), (\pm 1, \pm 6),$\\ & $(\pm 4, \pm 3), (\pm 7, \pm 23), (\pm 9, \pm 2)$\\ \hline
$(13, 79)$ & $(-2,-5),(2,5)$\\ \hline

\multirow{2}{3.5em}{$(7,\pm 83)$} & $(\mp 8, \mp 13), (\mp 7, \mp 1), (\mp 6, \mp 19),$\\
&  $(\pm 3, \pm 11), (\pm 5, \pm 2), (\pm 13, \pm 20)$\\ \hline
$(11, \pm 89)$ & $(\mp 1, \pm 1) $\\ \hline
$(7,\pm 97)$ & $(\mp 4, \mp 11), (\mp 3, \pm 1), (\pm 7, \pm 10)$\\ \hline

\end{tabular}
\medskip
\captionof{table}{\textit{Solutions (assuming GRH) to the Thue equations where $D=\pm \ell$ and $41\leq \ell \leq 97$ \cite{BCOT}.}}
\label{thueGRHtable}
\end{table}
\end{center}
\endgroup

\begingroup
\setlength{\tabcolsep}{5pt} 
\renewcommand{\arraystretch}{1.9}
\begin{center} 
\begin{table}[!ht]
\begin{tabular}{|c|c|}
\multicolumn{2}{c}{} \\ \hline
$\ell$ & Integer Solutions to $C^-_4:y^4- 3x^2y^2+x^4=\ell$  \\ \hline \hline
$5 $ & $(\pm 1,-2),(\pm 2,-1),(\pm 2,1),(\pm 1,2)  $\\ \hline
$31$ & $(\pm 3,-5),(\pm 5,-3),(\pm 5,3),(\pm 3,5)$\\ \hline
$11,19,29,41,59,61,71,79,89 $ & $\emptyset$\\ \hline

\end{tabular}
\medskip
\captionof{table}{\textit{}}
\label{C_4^-table}
\end{table}
\end{center}
\endgroup

\begingroup
\setlength{\tabcolsep}{5pt} 
\renewcommand{\arraystretch}{1.9}
\begin{center} 
\begin{table}[!ht]
\begin{tabular}{|c|c|}
\multicolumn{2}{c}{} \\ \hline
$\ell$ & Integer Solutions to $C^-_4:y^4- 3x^2y^2+x^4=\ell$  \\ \hline \hline
$-11 $ & $(\pm 2,-3),(\pm 3,-2),(\pm 3,2),(\pm 2,3)$ \\ \hline
$-79 $ & $(\pm 5,-8),(\pm 8,-5),(\pm 8,5),(\pm 5,8)$\\ \hline
$-5,-19,-29,-31,-41,-59,-61,-71,-89 $ & $\emptyset$\\ \hline

\end{tabular}
\medskip
\captionof{table}{\textit{}}
\label{negativeC_4^-table}
\end{table}
\end{center}
\endgroup

\begingroup
\setlength{\tabcolsep}{5pt} 
\renewcommand{\arraystretch}{1.9}
\begin{center} 
\begin{table}[!ht]
\begin{tabular}{|c|c|}
\multicolumn{2}{c}{} \\ \hline
$\ell$ & Integer Solutions to $C^+_4:y^4+ 3x^2y^2+x^4=\ell$  \\ \hline \hline
$5 $ & $(\pm 1,-1),(\pm 1,1)$\\ \hline
$29$ & $(\pm 1,-2),(\pm 2,-1),(\pm 2,1),(\pm 1,2)$\\ \hline
$11,19,31,41,59,61,71,79,89 $ & $\emptyset$\\ \hline
\end{tabular}
\medskip
\captionof{table}{\textit{}}
\label{C_4^+table}
\end{table}
\end{center}
\endgroup

\begingroup
\setlength{\tabcolsep}{5pt} 
\renewcommand{\arraystretch}{1.9}
\begin{center} 
\begin{table}[!ht]
\begin{tabular}{|c|c|}
\multicolumn{2}{c}{} \\ \hline
$\ell$ & Integer Solutions to $y^4-3xy^2+x^2=\ell$  \\ \hline \hline
$11$ & $(-2,\pm 1),(5,\pm 1)$\\ \hline
$-19$ & $(5,\pm 2),(7,\pm 2)$\\ \hline
$29$ & $(-4,\pm 1),(-1,\pm 2),(7,\pm 1),(13,\pm 2)$\\ \hline
$-31$ & $(7,\pm 4),(19,\pm 7),(41,\pm 4),(128,\pm 7)$\\ \hline
$41$ & $(-5,\pm 1),(5,\pm 4),(8,\pm 1),(43,\pm 4)$\\ \hline
$59$ & $(46,\pm 11),(317,\pm 11)$\\ \hline
$-61$ & $\emptyset$\\ \hline
$71$ & $(-7,\pm 1),(10,\pm 1),(202,\pm 23),(1385,\pm 23)$\\ \hline
$-79$ & $(11,\pm 5),(25,\pm 8),(64,\pm 5),(167,\pm 8)$\\ \hline
$89$ & $(-8,\pm 1),(8,\pm 5),(11,\pm 1),(67,\pm 5)$\\ \hline
\end{tabular}
\medskip
\captionof{table}{\textit{} }
\label{Weight2}
\end{table}
\end{center}
\endgroup

\end{document}